\documentclass[final]{siamltex}
\usepackage{latexsym,amssymb,amsmath,amsfonts}%,amsthm}
\usepackage{mathrsfs}
\usepackage{graphicx}
\newcommand{\R}{{\mathbb R}}

\newcommand{\ds}{\displaystyle}
\newcommand{\no}{\nonumber}
\newcommand{\be}{\begin{eqnarray}}
\newcommand{\ben}{\begin{eqnarray*}}
\newcommand{\en}{\end{eqnarray}}
\newcommand{\enn}{\end{eqnarray*}}
\newcommand{\ba}{\backslash}
\newcommand{\pa}{\partial}

\newcommand{\ov}{\overline}

\newcommand{\g}{\gamma}
\newcommand{\G}{\Gamma}

\newcommand{\Om}{\Omega}
\newcommand{\om}{\omega}

\newcommand{\al}{\alpha}
\newcommand{\la}{\lambda}

\newcommand{\wi}{\widetilde}

\newtheorem{remark}[theorem]{Remark}

\begin{document}
\renewcommand{\theequation}{\arabic{section}.\arabic{equation}}
%\begin{titlepage}
\title{\bf Direct and inverse obstacle scattering problems in a piecewise homogeneous medium}
\author{Xiaodong Liu\thanks{LSEC and Institute of Applied Mathematics, AMSS,
Chinese Academy of Sciences,Beijing 100190, China
({\tt lxd230@163.com}).}
\and
Bo Zhang\thanks{LSEC and Institute of Applied Mathematics,
AMSS, Chinese Academy of Sciences, Beijing 100190, China
({\tt b.zhang@amt.ac.cn}).}
}
%\date{}
%\end{titlepage}
\maketitle
%\vspace{.2in}

\begin{abstract}
This paper is concerned with the problem of scattering of time-harmonic acoustic waves from
an impenetrable obstacle in a piecewise homogeneous medium. The well-posedness of the direct
problem is established, employing the integral equation method and then used, in conjunction
with the representation in a combination of layer potentials of the solution, to prove a priori
estimates of solutions on some part of the interface between the layered media.
The inverse problem is also considered in this paper. An uniqueness result is obtained for the
first time in determining both the penetrable interface and the impenetrable obstacle
with its physical property from a knowledge of the far field pattern for incident plane waves.
In doing so, an important role is played by the a priori estimates of the solution for the
direct problem.

%\vspace{.2in}

\begin{keywords}
%{\bf Keywords:}
Uniqueness, piecewise homogeneous medium, acoustic, Holmgren's uniqueness theorem,
inverse scattering.
\end{keywords}

\begin{AMS}
35P25, 35R30
\end{AMS}

\end{abstract}

\pagestyle{myheadings}
\thispagestyle{plain}
\markboth{X. LIU, B. ZHANG}{Inverse scattering in piecewise homogeneous medium}

\section{Introduction}
\setcounter{equation}{0}

In this paper, we consider the problem of scattering of
time-harmonic acoustic plane waves by an impenetrable obstacle
in a piecewise homogeneous medium. In practical
applications, the background might not be homogeneous and then may
be modeled as a layered medium. A medium of this type that is a
nested body consisting of a finite number of homogeneous layers
occurs in various areas of applications such as radar, remote
sensing, geophysics, and nondestructive testing.

%For simplicity, and without loss of generality, in this paper
%we restrict ourself to the case where the
%obstacle is buried in a two-layered background medium, as shown
%in Figure \ref{fig1}. In particular
To give a precise description of the problem, let $\Om_{2}\subset\R^3$
denote the impenetrable obstacle which is an open bounded region
with a $C^{2}$ boundary $S_{1}$ and let $\R^3\ba\ov{\Om_{2}}$ denote the
the background medium which is divided by means of a closed $C^{2}$
surface $S_{0}$ into two connected domains $\Om_{0}$ and $\Om_{1}$ (see Figure \ref{fig1}).
Here, $\Om_{0}$ is the unbounded homogeneous medium and $\Om_{1}$ is the bounded homogeneous one.
We assume that the boundary $S_{1}$ of the obstacle $\Om_{2}$ has a
dissection $S_{1}=\ov{\G}_{0}\cup\ov{\G}_{1}$, where $\G_{0}$ and $\G_{1}$
are two disjoint, relatively open subsets of $S_{1}$.
Furthermore, the Dirichlet and impedance boundary conditions with
the surface impedance a nonnegative continuous function $\la\in C(\G_{1})$
are specified on $\G_{0}$ and $\G_{1}$, respectively.
Note that the case $\G_{1}=\emptyset$ corresponds to
a {\em sound-soft} obstacle and the case $\G_{0}=\emptyset,\;\la=0$
leads to a Neumann boundary condition which corresponds to a {\em sound-hard} obstacle.

The scattering of time-harmonic acoustic waves in a two-layered
medium in $\R^3$ is now modeled by the Helmholtz equation with
boundary conditions on the interface $S_{0}$ and boundary $S_{1}$:
 \be
 \label{0HE0}\Delta u+k_{0}^{2}u=0&&\qquad \mbox{in}\ \Om_{0},\\
 \label{0HE1}\Delta v+k_{1}^{2}v=0&&\qquad \mbox{in}\ \Om_{1},\\
 \label{0tbc}u-v=0,\;\;\frac{\pa u}{\pa\nu}-\la_0\frac{\pa v}{\pa\nu}=0&&\qquad\mbox{on}\;S_0,\\
 \label{Bv}\mathscr{B}(v)=0&&\qquad \mbox{on}\; S_{1},\\
 \label{0rc}\lim_{r\rightarrow\infty}r(\frac{\pa u^{s}}{\pa r}-ik_{0}u^{s})=0&& r=|x|
 \en
where $\nu$ is the unit outward normal to the interface $S_{0}$ and boundary $S_{1}$,
$\la_0$ is a positive constant. Here, the total field $u=u^{s}+u^{i}$ is given as
the sum of the unknown scattered wave $u^{s}$ which is required to satisfy the
Sommerfeld radiation condition (\ref{0rc}) and incident plane wave $u^{i}=e^{ik_{0}x\cdot d}$,
where $k_{j}$ is the positive wave number given by $k_{j}=\om_{j}/c_{j}$ in terms of the frequency
$\om_{j}$ and the sound speed $c_{j}$ in the corresponding region $\Om_{j}\;(j=0,1)$.
The distinct wave numbers $k_{j}\;(j=0,1)$ correspond to the fact that
the background medium consists of two physically different materials.
On the interface $S_{0}$, the so-called "transmission condition"
(\ref{0tbc}) is imposed, which represents the continuity of the medium
and equilibrium of the forces acting on it.
The boundary condition $\mathscr{B}(v)=0$ on $S_{0}$ is understood as:
 \be
 \label{0dbc}v=0&&\qquad\mbox{on}\ \G_{0},\\
 \label{0ibc}\frac{\pa v}{\pa\nu}+i\la v=0&&\qquad \mbox{on}\ \G_{1}.
 \en
Thus, the boundary condition (\ref{Bv}) is a general and realistic one and
allows that the pressure of the total wave $v$ vanishes on $\G_{0}$ and the normal
velocity is proportional to the excess pressure on the coated part $\G_{1}$.
\begin{figure}[htbp]
%\centering
\begin{center}
\scalebox{0.35}{\includegraphics{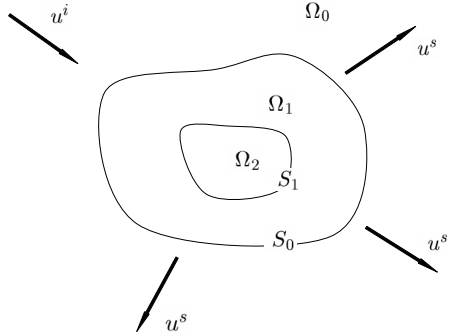}}
\caption{Scattering in a two-layered background medium }\label{fig1}
\end{center}
\end{figure}

The {\em direct problem} is to seek a pair of functions
$u\in C^{2}(\Om_{0})\cap C^{1,\al}(\ov{\Om_{0}})$
and $v\in C^{2}(\Om_{1})\cap C^{1,\al}(\ov{\Om_{1}})$ satisfying (\ref{0HE0})-(\ref{0rc}).
By the variational method, the well-posedness (existence, uniqueness and stability) of the
direct problem has been established in \cite{AS96} for the Dirichlet boundary condition
and in \cite{LZH} for a general mixed boundary condition (\ref{Bv}).
In the present paper, an integral equation method is employed to establish
the well-posedness of the direct problem. This result is also used, in conjunction with the
representation in a combination of layer potentials of the solution, to prove a priori
estimates of the solution on some part of the interface $S_{0}$, which plays an important role
in the proof of the uniqueness result for our inverse problem later on.

Further, it is known that $u^{s}(x)$ has the following asymptotic representation
 \be\label{1.5}
u^{s}(x,d)=\frac{e^{ik_{0}|x|}}{|x|}\left\{u^{\infty}(\widehat{x},d)
+O(\frac{1}{|x|})\right\}\; \mbox{as }\;|x|\rightarrow\infty
 \en
uniformly for all directions $\widehat{x}:=x/|x|$, where the
function $u^{\infty}(\widehat{x},d)$ defined on the unit sphere $S$
is known as the far field pattern with $\widehat{x}$ and $d$
denoting, respectively, the observation direction and the incident direction.

The {\em inverse problem} we consider in this paper is, given the wave
numbers $k_{j}$ ($j=0,1$), the positive constant $\la_{0}$ and the far field pattern
$u^{\infty}(\widehat{x},d)$ for all incident plane waves with
incident direction $d\in S$, to determine the obstacle $\Om_{2}$ with
its physical property $\mathscr{B}$ and the interface $S_{0}$.
As usual in most of the inverse problems, the first question to ask in this
context is the identifiability, that is, whether an inaccessible obstacle $\Om_{2}$
with its physical property $\mathscr{B}$ and the interface $S_{0}$ can be
identified from a knowledge of the far-field pattern.
Mathematically, the identifiability is the uniqueness issue which is of theoretical
interest and is required in order to proceed to efficient numerical methods of solutions.

Since the first uniqueness result given by Schiffer in 1967 for a {\em sound-soft} obstacle \cite{CK,LP},
there has been an extensive study in this direction in the literature; see, e.g.
\cite{CK06,CS83,Gintides,GK96,Isakov90,Isakov,Po01,Ramm,Ramm93,Ramm95,Rondi03,Sleeman82,SU04}
for scattering in a homogeneous medium, \cite{H98,KP,N07} for scattering
in an inhomogeneous medium and \cite{LZou} for scattering by special obstacles
such as balls and polyhedra.
%Recently, there are also fruitful results for special obstacles such as balls and polyhedra,
%we refer to \cite{LZou} and \cite{Lxdreview} for a review.
However, there are few uniqueness results for inverse obstacle scattering in a
piecewise homogeneous medium. For the case of a known piecewise homogeneous medium,
Yan and Pang \cite{YP} established a uniqueness result for the inverse scattering problem
of determining a {\em sound-soft} obstacle based on Schiffer's idea;
their method can not be extended to other boundary conditions.
They obtained a uniqueness result for the case of a {\em sound-hard} obstacle in a
two-layered background medium in \cite{PY} using a generalization of Schiffer's method.
However, their method is hard to be extended to the case of a multilayered background medium
and seems unreasonable to require the interior wave number to be in an interval.
Recently in \cite{LZH}, based on a generalization of the mixed reciprocity relation,
we proved that both the obstacle $\Om_{2}$ and its physical property $\mathscr{B}$ can be
uniquely recovered from a knowledge of the far field pattern for incident plane waves.
This seems to be appropriate for a number of applications where the
physical nature of the obstacle is unknown.
The tools and the uniqueness result developed in \cite{LZH} can also be extended to inverse
electromagnetic scattering problems \cite{LZ}.
For the case of an unknown piecewise homogeneous medium,
Athanasiadis, Ramm and Stratis \cite{ARS} and Yan \cite{YG} proved that the interfaces
between the layered media can be determined uniquely by the corresponding far field pattern
in the special case when the impenetrable obstacle does not exist.

However, to the authors' knowledge, no uniqueness result is available for determining
both the obstacle embedded in the piecewise homogeneous medium and the interfaces
between the layered media from a knowledge of the far field pattern for incident plane waves.
In this paper, we will prove for the first time that both the inaccessible
obstacle $\Om_{2}$ with its physical property $\mathscr{B}$ and the interface $S_{0}$
can be uniquely determined by a knowledge of the far-field pattern.
%Our method can be extended straightforwardly to the case of a multilayered medium
%and the 2D case, and the same results as in this paper are also valid for these two cases.
We remark that the results obtained in this paper are also available for both the 2D case and
the case of a multilayered medium and can be proved similarly.

The remaining part of the paper is organized as follows. In the next section,
we will establish the well-posedness of the direct scattering problem, employing the
integral equation method. With the help of the representation in a combination of
layer potentials of the solution, a priori estimates of solutions are also obtained
on some part of the interface between the layered media.
Section \ref{sec3} is devoted to the proof of the result on the unique determination
of both the obstacle $\Om_2$ with its physical property $\mathscr{B}$ and the surface
$S_0$ from a knowledge of the far field pattern for incident plane waves.

%%%%%%%%%%%%%%%%%%%%%%%%%%%%%%%%%%%%%%%%%%%%%%%%%%%%%%%%%%%%%%%%%%%%%%%%%%%%%%%%%%%%%%%%%%%%%%%%%%%%%
%%%%%%%%%%%%%%%%%%%%%%%%%%%%%%%%%%%%%%%%%%%%%%%%%%%%%%%%%%%%%%%%%%%%%%%%%%%%%%%%%%%%%%%%%%%%%%%%%%%%%

\section{The direct scattering problem}\label{sec2}
\setcounter{equation}{0}

%For the investigation of inverse scattering problems a good knowledge about the direct
%scattering problems is essential. Thus in this section on direct scattering problem
In this section we first establish the well-posedness of the direct problem,
employing the integral equation method and then make use of the representation in a
combination of layer potentials of the solution to prove some a priori estimates of the solution
which plays an important role in the inverse problem.
We shall use $C$ to denote a generic constant whose values may change in different
inequalities but always bounded away from infinity.

As incident fields $u^{i}$, plane waves and point sources (cf. (\ref{Phidef}) below) are of
special interest. Denote by $u^{s}(\cdot,d)$ the scattered field for an incident plane wave
$u^{i}(\cdot,d)$ with incident direction $d\in S$ and by $u^{\infty}(\cdot,d)$ the corresponding
far field pattern. The scattered field for an incident point source $\Phi(\cdot,z)$ with source
point $z\in \R^3$ is denoted by
$u^{s}(\cdot; z)$ and the corresponding far field pattern by
$\Phi^{\infty}(\cdot,z)$.

The direct problem is to look for a pair of functions $u\in C^{2}(\Om_{0})\cap C^{1,\al}(\ov{\Om_{0}})$
and $v\in C^{2}(\Om_{1})\cap C^{1,\al}(\ov{\Om_{1}})$ satisfying the following boundary value problem:
 \be\label{HE0}
\Delta u+k_{0}^{2}u=0&&\qquad \mbox{in}\ \Om_{0},\\ \label{HE1}
\Delta v+k_{1}^{2}v=0&&\qquad \mbox{in}\ \Om_{1},\\ \label{tbc}
u-v=f,\;\;\frac{\pa u}{\pa\nu}-\la_0\frac{\pa v}{\pa\nu}=g&&\qquad\mbox{on}\;S_0,\\ \label{dbc}
v=0&&\qquad\mbox{on}\ \G_{0},\\  \label{ibc}
\frac{\pa v}{\pa\nu}+i\la v=0&&\qquad \mbox{on}\ \G_{1},\\  \label{rc}
\lim_{r\rightarrow\infty}r(\frac{\pa u}{\pa r}-ik_{0}u)=0,&& r=|x|.
 \en
Here, we assume that $k_{0}, k_{1}$ and $\la_{0}$ are given positive constants and
that $f\in C^{1,\al}(S_{0})$ and $g\in C^{0,\al}(S_{0})$ are given functions in
H\"{o}lder spaces with exponent $0<\al<1$.

\begin{remark}\label{r2} {\rm
The problem of scattering of the incident plane wave $\ds u^{i}=e^{ik_{0}x\cdot d}$
is a particular case of the problem (\ref{HE0})-(\ref{rc}). In particular,
the scattered field $u^{s}$ satisfies the problem (\ref{HE0})-(\ref{rc}) with
$\ds u=u^s,\;f=-u^{i}|_{S_{0}},\;g=-\frac{\pa u^{i}}{\pa\nu}|_{S_{0}}$.
}
\end{remark}

The following uniqueness result has been established in \cite{LZH} (see Theorem 2.3 of \cite{LZH}).

\begin{theorem}\label{uni.direct}
The boundary value problem $(\ref{HE0})-(\ref{rc})$ admits at most one solution.
\end{theorem}

%\begin{proof}
%We refer to \cite{LZH} for a detailed proof.
%\end{proof}

Denote by $\Phi_{j}$ the fundamental solution of the Helmholtz equation
with wave number $k_{j}\;(j=0,1)$, which is given by
 \be\label{Phidef}
 \Phi_{j}(x,y)=\frac{e^{ik_{j}|x-y|}}{4\pi |x-y|}, \qquad \ x,y\in \R^{3},x\neq y.
 \en
For $i,j=0,1$ define the single- and double-layer operators $S_{i,j}$ and $K_{i,j}$, respectively, by
 \ben
 (S_{i,j}\phi)(x):= \int_{S_{0}}\Phi_{j}(x,y)\phi(y)ds(y)\qquad &\;x\in S_{i},\\
 (K_{i,j}\phi)(x):= \int_{S_{0}}\frac{\pa\Phi_{j}(x,y)}{\pa\nu(y)}\phi(y)ds(y)\qquad &\;x\in S_{i}
 \enn
and the normal derivative operators $K^{'}_{i,j}$ and $T_{i,j}$ by
 \ben
 (K^{'}_{i,j}\phi)(x):=\frac{\pa}{\pa\nu(x)}\int_{S_{0}}\Phi_{j}(x,y)\phi(y)ds(y)\qquad &\;x\in S_{i},\\
 (T_{i,j}\phi)(x):=\frac{\pa}{\pa\nu(x)}
      \int_{S_0}\frac{\pa\Phi_j(x,y)}{\pa\nu(y)}\phi(y)ds(y)\qquad&\;x\in S_{i}.
 \enn
For $i,j=0,1$ define the single- and double-layer operators $\wi{S}_{i,j}$
and $\wi{K}_{i,j}$, respectively, by
 \be\label{wiSij}
 (\wi{S}_{i,j}\phi)(x):= \int_{\G_{j}}\Phi_{1}(x,y)\phi(y)ds(y)\qquad &\;x\in S_{i},\\ \label{wiKij}
 (\wi{K}_{i,j}\phi)(x):= \int_{\G_{j}}\frac{\pa\Phi_{1}(x,y)}{\pa\nu(y)}\phi(y)ds(y)\qquad &\;x\in S_{i}
 \en
and the normal derivative operators $\wi{K}^{'}_{i,j}$ and $\wi{T}_{i,j}$ by
 \ben
 (\wi{K}^{'}_{i,j}\phi)(x):=\frac{\pa}{\pa\nu(x)}
   \int_{\G_{j}}\Phi_{1}(x,y)\phi(y)ds(y)\qquad &\; x\in S_{i},\\
 (\wi{T}_{i,j}\phi)(x):= \frac{\pa}{\pa\nu(x)}
   \int_{\G_{j}}\frac{\pa\Phi_{1}(x,y)}{\pa \nu(y)}\phi(y)ds(y)\qquad&\; x\in S_{i}.
 \enn

For the mapping properties of these operators in the spaces of continuous and H\"{o}lder continuous
functions we refer to Section 3.1 in \cite{CK} or Chapter 2 in \cite{CK83}.

\begin{theorem}\label{well-posedness}
The boundary value problem $(\ref{HE0})-(\ref{rc})$ has a unique solution.
Moreover, there exists a positive constant $C=C(\la_0, \Om_1, \al)$ such that
\be\label{estimate}
\|u\|_{C^{1,\al}(\ov{\Om_{0}})}+\|v\|_{C^{1,\al}(\ov{\Om_{1}})}\leq
 C(\|f\|_{C^{1,\al}(S_{0})}+\|g\|_{C^{0,\al}(S_{0})}).
\en
\end{theorem}

\begin{proof}
The uniqueness of solutions follows from Theorem \ref{uni.direct}.
We now prove the existence of solutions by using the integral equation method.
Following \cite{CK83} and \cite{CK} we seek a solution in the form
\be\label{u}
u(x)&=&\int_{S_{0}}\left\{\la_0\frac{\pa\Phi_{0}(x,y)}{\pa\nu(y)}\psi(y)
     +\Phi_{0}(x,y)\phi(y)\right\}ds(y)\hspace{0.6cm}\qquad x\in\Om_{0},\qquad\\ \no
v(x)&=&\int_{S_{0}}\left\{\frac{\pa\Phi_{1}(x,y)}{\pa\nu(y)}\psi(y)
       +\Phi_{1}(x,y)\phi(y)\right\}ds(y)\\ \no
 &&+\int_{\G_0}\left\{\frac{\pa\Phi_1(x,y)}{\pa\nu(y)}\chi(y)
       -i\eta\Phi_1(x,y)\chi(y)\right\}ds(y)\\ \label{v}
 &&+\int_{\G_{1}}\left\{\Phi_{1}(x,y)\varphi(y)
   +i\eta\frac{\pa\Phi_{1}(x,y)}{\pa\nu(y)}(\wi{S}^{2}\varphi)(y))
   \right\}ds(y)\qquad x\in\Om_{1} %\cup\Om_{2}
\en
with four densities $\psi\in C^{1,\al}(S_{0})$, $\phi\in C^{0,\al}(S_{0})$,
$\chi\in C^{1,\al}(\G_{0})$, $\varphi\in C^{1,\al}(\G_{1})$ and a real
coupling parameter $\eta\neq0$. By $\wi{S}$ we denote the single-layer
operator (\ref{wiSij}) in the potential theoretic limit case $k_{1}=0$.
Then from the jump relations we see that the potentials $u$ and $v$ defined above
solve the boundary value problem $(\ref{HE0})-(\ref{rc})$ provided
the densities $\psi,\phi,\chi,\varphi$ satisfy the system of integral equations
\be\no
&&\psi+\mu(\la_0K_{0,0}-K_{0,1})\psi+\mu(S_{0,0}-S_{0,1})\phi\\ \label{psi}
&&\qquad\qquad\qquad\qquad\qquad-\mu(\wi{K}_{0,0}-i\eta\wi{S}_{0,0})\chi+M_1\varphi
  =\mu f\quad\mbox{on}\;\;S_0,\\ \no
&&\phi-\mu\la_0(T_{0,0}-T_{0,1})\psi-\mu(K^{'}_{0,0}-\la_0K^{'}_{0,1})\phi\\ \label{phi}
&&\qquad\qquad\qquad\qquad-\mu\la_0(T_{0,0}-i\eta K^{'}_{0,0})\chi+M_2\varphi
  =-\mu g\quad\mbox{on}\;\;S_0,\\ \label{chi}
&&\chi+2K_{1,1}\psi+2S_{1,1}\phi+2(\wi{K}_{1,0}-i\eta \wi{S}_{1,0})\chi
  +M_3\varphi=0\quad\mbox{on}\;\;\G_0,\\ \label{varphi}
&&\varphi-2(T_{1,1}+i\la K_{1,1})\psi-2(K^{'}_{1,1}+i\la S_{1,1})\phi
  +M_4\chi+M_5\varphi=0\quad\mbox{on}\;\;\G_1
\en
with $\mu={2}/({\la_0+1})$, where
\ben
M_1&=&-\mu(\wi{S}_{0,1}+i\eta\wi{K}_{0,1}\wi{S}^{2}),\\
M_2&=&-\mu(\wi{K}^{'}_{0,1}+i\eta\wi{T}_{0,1}\wi{S}^{2}),\\
M_3&=&2(\wi{S}_{1,1}+i\eta\wi{K}_{1,1}\wi{S}^{2}),\\
M_4&=&-2(\wi{T}_{1,0}-i\eta\wi{K}^{'}_{1,0}+i\la\wi{K}_{1,0}+\la\eta\wi{S}_{1,0}),\\
M_5&=&-2(\wi{K}_{1,1}+i\eta\wi{T}_{1,1}\wi{S}^2-\la\eta\wi{S}^2
    +i\la\wi{S}_{1,1}-\la\eta\wi{K}_{1,1}\wi{S}^2).
\enn

Define the product space
$X:=C^{1,\al}(S_{0})\times C^{0,\al}(S_{0})\times C^{1,\al}(\G_{0})\times C^{1,\al}(\G_{1})$
and introduce the operator $A:X\rightarrow X$ given by
 \ben
 A:=\left(\begin{matrix}\mu(\la_0K_{0,0}-K_{0,1})&\mu(S_{0,0}-S_{0,1})
        &-\mu(\wi{K}_{0,0}-i\eta\wi{S}_{0,0})&M_1\\
      -\mu\la_0(T_{0,0}-T_{0,1})&-\mu(K^{'}_{0,0}-\la_0K^{'}_{0,1})
      &-\mu\la_0(T_{0,0}-i\eta K^{'}_{0,0})&M_2\\
      2K_{1,1}&2S_{1,1}&2(\wi{K}_{1,0}-i\eta\wi{S}_{1,0})&M_3\\
      -2(T_{1,1}+i\la K_{1,1})&-2(K^{'}_{1,1}+i\la S_{1,1})&M_4&M_5
 \end{matrix}\right)
 \enn
%with $M_1=-\mu(\wi{S}_{0,1}+i\eta\wi{K}_{0,1}\wi{S}^{2}),$
%$M_2=-\mu(\wi{K}^{'}_{0,1}+i\eta\wi{T}_{0,1}\wi{S}^{2}),$
%$M_3=2(\wi{S}_{1,1}+i\eta\wi{K}_{1,1}\wi{S}^{2}),$
%$M_4=-2(\wi{T}_{1,0}-i\eta\wi{K}^{'}_{1,0}+i\la\wi{K}_{1,0}+\la\eta\wi{S}_{1,0})$
%and $M_5=-2(\wi{K}_{1,1}+i\eta\wi{T}_{1,1}\wi{S}^2-\la\eta\wi{S}^2
%    +i\la\wi{S}_{1,1}-\la\eta\wi{K}_{1,1}\wi{S}^2).$
The operator $A$ is compact since all its entries are compact.
The system (\ref{psi})-(\ref{varphi}) can be rewritten in the abbreviated form
\be\label{matrix-form}
 (I+A)U=R,
\en
where $I$ is the identity operator, $U=(\psi,\phi,\chi,\varphi)^{T}$ and $R=(\mu f, -\mu g, 0, 0)^{T}$.
Thus, the Riesz-Fredholm theory is applicable.  We now prove the uniqueness of solutions to
the system (\ref{matrix-form}).
To this end, let $U$ be a solution of the homogeneous system corresponding to (\ref{matrix-form})
(that is, the system (\ref{matrix-form}) with $R=0$). Then it is enough to show that $U=0$.
%The proof is broken down into two steps.

We first prove that $\chi=0$ on $\G_{0}$ and $\varphi=0$ on $\G_{1}$.
From the system (\ref{matrix-form}) or (\ref{psi})-(\ref{varphi}) with $\mu f=-\mu g=0$ (since $R=0$)
it is known that $u$ and $v$ defined in (\ref{u}) and (\ref{v}) satisfy the problem
$(\ref{HE0})-(\ref{rc})$ with $f=g=0$. Thus, by the uniqueness Theorem \ref{uni.direct},
$u=0$ in $\Om_0$ and $v=0$ in $\Om_1$. Note that $v$, given by (\ref{v}), can also be defined for
$x\in\Om_2$ and satisfies the Helmholtz equation $\Delta v+k_{1}^{2}v=0$ in $\Om_2$.
Then the jump relations yield that
 \ben
 -v_{-}=\chi, \; -\frac{\pa v_{-}}{\pa\nu}=i\eta\chi \qquad \mbox{on}\; \G_{0},\\
 -v_{-}=i\eta\wi{S}^{2}\varphi, \; -\frac{\pa v_{-}}{\pa\nu}=-\varphi \qquad \mbox{on}\; \G_{1}.
 \enn
Interchanging the order of integration and using Green's first theorem over $\Om_2$, we obtain
\ben
i\eta\left\{\int_{\G_{0}}|\chi|^{2}ds+\int_{\G_{1}}|\wi{S}\varphi|^{2}ds\right\}
 &=&i\eta\left\{\int_{\G_{0}}|\chi|^{2}ds
    +\int_{\G_{1}}\varphi\wi{S}^{2}\ov{\varphi}ds\right\}\\
 &=&\int_{\G_0}\ov{v_{-}}\frac{\pa v_{-}}{\pa\nu}ds
    +\int_{\G_{1}}\ov{v}_{-}\frac{\pa v_{-}}{\pa\nu}ds\\
 &=&\int_{\Om_2}{|\nabla v|^{2}-k^{2}_{1}|v|^{2}}dx.
\enn
Taking the imaginary part of this equation gives that $\chi=0$ on $\G_{0}$
and $\wi{S}\varphi=0$ on $\G_{1}$. The single-layer potential
 \ben
 w(x):=\int_{\G_{1}}\frac{1}{4\pi|x-y|}\varphi(y)ds(y)
 \enn
with density $\varphi$ is continuous throughout $\R^{3}$, harmonic in $\R^{3}\ba\G_{1}$ and
vanishes on $\G_{1}$ and at infinity. Therefore, by the maximum-minimum principle for harmonic
functions, we have $w=0$ in $\R^3$ and the jump relation yields $\varphi=0$.

Now the system (\ref{matrix-form}) becomes
\ben
 \left [I +\left ( \begin{matrix}\mu(\la_0K_{0,0}-K_{0,1})&\mu(S_{0,0}-S_{0,1})\\
      -\mu\la_0(T_{0,0}-T_{0,1})&-\mu(K^{'}_{0,0}-\la_0K^{'}_{0,1})
 \end{matrix} \right )\right ]\left ( \begin{matrix}\psi\\ \phi\end{matrix} \right )
 =\left ( \begin{matrix}0\\0\end{matrix} \right ).
\enn
Define
\ben
 \wi{v}(x)&:=&\int_{S_{0}}\left\{\frac{\pa\Phi_{1}(x,y)}{\pa\nu(y)}\psi(y)
          +\Phi_{1}(x,y)\phi(y)\right\}ds(y),\qquad x\in\Om_0,\\
 \wi{u}(x)&:=&-\int_{S_{0}}\left\{\frac{\pa\Phi_{0}(x,y)}{\pa\nu(y)}\psi(y)
          +\frac{1}{\la_0}\Phi_{0}(x,y)\phi(y)\right\}ds(y),\qquad x\in\R^3\ba\ov{\Om}_0.
 \enn
Then by the jump relations for single- and double-layer potentials we have
\be\label{htbc1}
\wi{v}-v=\psi,&&\;\;\;\frac{1}{\la_0}u+\wi{u}=\psi\qquad\mbox{on}\;\;S_0,\\ \label{htbc2}
\frac{\pa \wi{v}}{\pa\nu}-\frac{\pa v}{\pa\nu}=-\phi,&&\;\;\;
\frac{\pa u}{\pa\nu}+\la_0\frac{\pa\wi{u}}{\pa\nu}=-\phi\qquad\mbox{on}\;\;S_0.
\en
Hence, $\wi{v}$ and $\wi{u}$ solve the homogeneous transmission problem
 \ben
 \Delta \wi{v}+k_{1}^{2}\wi{v}=0\;\; \mbox{in}\;\; \Om_{0},\qquad
\Delta \wi{u}+k_{0}^{2}\wi{u}=0\;\; \mbox{in}\;\; \R^{3}\ba\ov{\Om}_0
 \enn
with the transmission conditions (noting that $u=0$ in $\Om_0$ and $v=0$ in $\Om_1$)
 \ben
\wi{v}-\wi{u}=0,\qquad\frac{\pa\wi{v}}{\pa\nu}=\la_0\frac{\pa\wi{u}}{\pa\nu}\;\;\mbox{on}\;\;S_0.
 \enn
Arguing similarly as in the proof of Theorem 2.3 in \cite{LZH} we can show that
$\wi{v}=0$ in $\Om_0$ and $\wi{u}=0$ in $\R^{3}\ba\ov{\Om_{0}}$.
Hence we conclude from (\ref{htbc1}) and (\ref{htbc2}) that $\psi=\phi=0$ on $S_0$.

Thus, the injectivity of the operator $I+A$ is proved, and by the Riesz-Fredholm theory
$(I+A)^{-1}$ exists and is bounded in $X$. From this we deduce that
\ben
&&\|\psi\|_{C^{1,\al}(S_{0})}+\|\phi\|_{C^{0,\al}(S_{0})}
   +\|\chi\|_{C^{1,\al}(\G_{0})}+\|\varphi\|_{C^{1,\al}(\G_{1})}\\
&&\qquad\qquad\qquad\leq C(\|f\|_{C^{1,\al}(S_{0})}+\|g\|_{C^{0,\al}(S_{0})}).
 \enn
The estimate (\ref{estimate}) follows from (\ref{u}), (\ref{v}) and Theorem 3.3 in \cite{CK}.
\end{proof}

We now make use of the representation (\ref{v}) of the solution $v$ to derive an a priori estimate of
the solution $v$ on some part of $S_{0}$, which is necessary in proving the
uniqueness result for the inverse problem in the next section.

Let $x^{\ast}\in S_{0}$ be an arbitrarily fixed point and let us introduce the space $C_{0}(S_{0})$
which consists of all continuous functions $h\in C(S_{0}\ba\{x^{\ast}\})$ with the property that
 \ben
\lim_{x\rightarrow x^{\ast}}[|x-x^{\ast}|h(x)]
 \enn
exists. It can easily be seen that $C_{0}(S_{0})$ is a Banach space equipped with
the weighted maximum norm
 \ben
 \|h\|_{C_{0}(S_{0})}:=\sup_{x\neq x^{\ast}, x\in S_{0}}|(x-x^{\ast})h(x)|.
 \enn
%
%All the operators
%$S_{0,j}, K_{0,j}, K^{'}_{0,j}, (j=0,1) T_{0,0}-T_{0,1}:C_{0}(S_{0})\rightarrow C_{0}(S_{0})$ are compact
%by Lemma 4.2 in \cite{KK93}.

\begin{lemma}\label{lemapriori}
Given two functions $f\in C^{1,\al}(S_{0})$ and $g\in C^{0,\al}(S_{0})$.
Let $u\in C^{2}(\Om_{0})\cap C^{1,\al}(\ov{\Om_{0}})$ and
$v\in C^{2}(\Om_{1})\cap C^{1,\al}(\ov{\Om_{1}})$
be the solution of the problem $(\ref{HE0})-(\ref{rc}).$
Let $x^{\ast}\in S_{0}$ and let $B_{1},$ $B_{2}$ be two small balls with center $x^{\ast}$
and radii $r_{1},\;r_{2}$, respectively, satisfying that $r_1<r_2$.
Then there exists a constant $C>0$ such that
\ben
&&\|v\|_{\infty, S_{0}\ba B_{2}}+\left\|\frac{\pa v}{\pa\nu}\right\|_{\infty, S_{0}\ba B_{2}}\\
&&\qquad\qquad\leq C(\|f\|_{C_{0}(S_{0})}+\|g\|_{C_{0}(S_{0})}
  +\|f\|_{1,\al,S_{0}\ba B_{1}}+\|g\|_{0,\al,S_{0}\ba B_{1}})
\enn
\end{lemma}

\begin{proof}
We consider again the system (\ref{matrix-form}) of the boundary integral equations derived from
the boundary value problem (\ref{HE0})-(\ref{rc}).
%again in the matrix-vector form (\ref{matrix-form}).
In addition to the space $X$, we also consider the weighted spaces
$C_{0}:=C_{0}(S_{0})\times C_{0}(S_{0})\times C^{1,\al}(\G_{0})\times C^{1,\al}(\G_{1})$.
The matrix operator $A$ is also compact in $C_0$ since all entries of $A$ are
compact (see \cite{KK93, KP}). From the proof of Theorem \ref{well-posedness},
we know that the operator $I+A$ has a trivial null space in $X$. Therefore,
by the Fredholm alternative applied to the dual system $\langle X, C_{0}\rangle$
with the $L^{2}$ bilinear form, the adjoint operator $I+A^\prime$
has a trivial null space in $C_{0}$. By the Fredholm alternative again, but now
applied to the dual system $\langle C_{0}, C_{0}\rangle$ with the $L^{2}$ bilinear form,
the operator $I+A$ also has a trivial null space in $C_{0}$.
Hence, by the Riesz-Fredholm theory, the system (\ref{matrix-form})
is also uniquely solvable in $C_{0}$,
and the solution depends continuously on the right-hand side:
\be\no
&&\|\psi\|_{C_{0}(S_{0})}+\|\phi\|_{C_{0}(S_{0})}+\|\chi\|_{C^{1,\al}(\G_{0})}
  +\|\varphi\|_{C^{1,\al}(\G_{1})}\\ \label{C0es}
&&\qquad\qquad\qquad\leq C(\|f\|_{C_{0}(S_{0})}+\|g\|_{C_{0}(S_{0})}).
\en
From (\ref{v}) and the jump relation we find that on $S_0,$
\be\label{vs0}
 v=-\frac{1}{2}\psi+K_{0,1}\psi+S_{0,1}\phi+(\wi{K}_{0,0}-i\eta \wi{S}_{0,0})\chi
    +(\wi{S}_{0,1}+i\eta\wi{K}_{0,1}\wi{S}^{2})\varphi.
\en
Thus we have
\be\no
\|v\|_{\infty, S_{0}\ba B_{2}}&\leq& C(\|\psi\|_{\infty, S_{0}\ba B_{2}}
      +\|K_{0,1}\psi\|_{\infty, S_{0}\ba B_{2}}+\|S_{0,1}\phi\|_{\infty, S_{0}\ba B_{2}}\\ \no
&&+\|(\wi{K}_{0,0}-i\eta \wi{S}_{0,0})\chi\|_{\infty, S_{0}\ba B_{2}}
    +\|(\wi{S}_{0,1}+i\eta\wi{K}_{0,1}\wi{S}^{2})\varphi\|_{\infty, S_{0}\ba B_{2}})\\ \no
&\leq& C(\|\psi\|_{C_{0}(S_{0})}+\|K_{0,1}\psi\|_{\infty, S_{0}\ba B_{2}}
    +\|S_{0,1}\phi\|_{\infty, S_{0}\ba B_{2}}\\ \label{ves1}
&&+\|\chi\|_{C^{1,\al}(\G_{0})}+\|\varphi\|_{C^{1,\al}(\G_{1})}).
\en
We choose a function $\rho_{1}\in C^{2}(S_{0})$ such that $\rho_1(x)=0$ for $x\in S_0\ba B_1$
and $\rho_{1}(x)=1$ in the neighborhood of $x^{\ast}$.
We also choose another function $\rho_{2}\in C^{2}(S_{0})$ such that $\rho_{2}(x)=1$
for $x\in S_{0}\ba B_{2}$ and $\rho_{2}(x)=0$ in the neighborhood of $B_{1}$.
Multiplying $K_{0,1}\psi$ by $\rho_{2}$ and splitting $\psi$ up in the form
$\psi=\rho_{1}\psi+(1-\rho_{1})\psi$, we have
\be\no
\|K_{0,1}\psi\|_{\infty, S_0\ba B_2}&=&\|\rho_{2}K_{0,1}\psi\|_{\infty, S_0\ba B_2}\\ \label{K01}
&\leq& C(\|\rho_{2}K_{0,1}\rho_{1}\psi\|_{\infty, S_{0}\ba B_{2}}
    +\|K_{0,1}(1-\rho_{1})\psi\|_{\infty, S_{0}\ba B_{2}})
\en
The first term on the right-hand side of the above inequality contains only an operator with a kernel
vanishing in a neighborhood of the diagonal $x=y$, and therefore we have
\be\label{K01-1}
 \|\rho_{2}K_{0,1}\rho_{1}\psi\|_{\infty, S_{0}\ba B_{2}}\leq C\|\psi\|_{C_{0}(S_{0})}.
\en
Since the operator $K_{0,1}$ mapping $C(S_{0})$ into $C^{0,\al}(S_{0})$ is bounded (see \cite{CK}),
we find, on noting that $1-\rho_{1}$ vanishes in a neighborhood of $x^{\ast}$, that
\be\label{K01-2}
\|K_{0,1}(1-\rho_1)\psi\|_{\infty,S_0\ba B_2}\leq C\|(1-\rho_1)\psi\|_{\infty,S_0\ba B_2}
\leq C\|\psi\|_{C_{0}(S_{0})}.
\en
From (\ref{K01})-(\ref{K01-2}) it follows that
\be\label{K01-3}
 \|K_{0,1}\psi\|_{\infty, S_{0}\ba B_{2}}\leq C\|\psi\|_{C_{0}(S_{0})}.
\en
A similar argument as above gives that
\be\label{S01-3}
 \|S_{0,1}\phi\|_{\infty, S_{0}\ba B_{2}}\leq C\|\phi\|_{C_{0}(S_{0})}.
\en
Combining (\ref{K01-3})-(\ref{S01-3}) with (\ref{C0es}) and (\ref{ves1}) yields
\be\label{ves}
 \|v\|_{\infty, S_{0}\ba B_{2}}\leq C(\|f\|_{C_{0}(S_{0})}+\|g\|_{C_{0}(S_{0})}).
\en

Before proceeding to estimate ${\pa v}/{\pa\nu}$
we establish the following estimate in the spaces of H\"{o}lder continuous functions
for $(\psi, \phi, \chi, \varphi)$:
\be\no
&&\|\psi\|_{C^{1,\al}(S_{0}\ba B_{3})}+\|\phi\|_{C^{0,\al}(S_{0}\ba B_{3})}
 +\|\chi\|_{C^{0,\al}(\G_{0})}+\|\varphi\|_{C^{0,\al}(\G_{1})}\\ \label{4-tuple}
&&\qquad\leq C(\|f\|_{C_{0}(S_{0})}+\|g\|_{C_{0}(S_{0})}
 +\|f\|_{1,\al,S_{0}\ba B_{1}}+\|g\|_{0,\al,S_{0}\ba B_{1}}),
\en
where $B_{3}$ is a ball of radius $r_{3}$ and centered at $x^*$ with $r_{1}<r_{3}<r_{2}$.
We choose a function $\rho_{3}\in C^{2}(S_{0})$ such that $\rho_{3}(x)=0$ for $x\in S_{0}\ba B_{2}$
and $\rho_{3}(x)=1$ in the neighborhood of $B_{3}$.
We also choose another function $\rho_{4}\in C^{2}(S_{0})$ such that $\rho_{4}(x)=1$
for $x\in S_{0}\ba B_{2}$ and $\rho_{4}(x)=0$ in the neighborhood of $B_{3}$.
%Similar to $\rho_{1}(x)\rho_{2}(y)$, $\rho_{3}(x)\rho_{4}(y)$ vanish in
%some neighborhood of the diagonal $x=y$.
Splitting $U$ up in the form
\ben
 U=\left ( \begin{matrix}\rho_{3}\psi\\  \rho_{3}\phi\\  \chi\\  \varphi \end{matrix} \right )
   +\left ( \begin{matrix}(1-\rho_{3})\psi\\  (1-\rho_{3})\phi\\  0\\  0 \end{matrix} \right )
   :=U_{\rho_{3}}+U_{(1-\rho_{3})}
\enn
and using $W_{\rho_{4}}$ to denote the matrix $W$ with its first and second rows
multiplied by $\rho_{4}(x)$, it follows from (\ref{matrix-form}) that
\be\label{rhoU}
 U_{\rho_{4}}=R_{\rho_{4}}-A_{\rho_{4}}U_{\rho_{3}}-A_{\rho_{4}}U_{(1-\rho_{3})}.
\en
Arguing similarly as in deriving the estimate for $\|v\|_{\infty,S_0\ba B_2}$
but with two different cutoff functions $\rho_{3}$ and $\rho_{4}$
replacing $\rho_1$ and $\rho_2$, we obtain from (\ref{rhoU}) and (\ref{C0es}) that
\be\no
\|U_{\rho_{4}}\|_{0,\al}&=&\|\psi\|_{C^{0,\al}(S_{0}\ba B_{3})}+\|\phi\|_{C^{0,\al}(S_{0}\ba B_{3})}
 +\|\chi\|_{C^{0,\al}(\G_{0})}+\|\varphi\|_{C^{0,\al}(\G_{1})}\\ \no
&\leq& C(\|R_{\rho_{4}}\|_{0,\al}+\|A_{\rho_{4}}U_{\rho_{3}}\|_{0,\al}
 +\|A_{\rho_{4}}U_{(1-\rho_{3})}\|_{0,\al})\\ \no
&\leq& C(\|R_{\rho_{4}}\|_{0,\al}+\|U\|_{C_{0}})\\ \label{U0al}
&\leq& C(\|f\|_{C_{0}(S_{0})}+\|g\|_{C_{0}(S_{0})}
 +\|f\|_{1,\al,S_{0}\ba B_{1}}+\|g\|_{0,\al,S_{0}\ba B_{1}}),
\en
where $\|\cdot\|_{0,\al}$ and $\|\cdot\|_{C_{0}}$ denote the corresponding norms in the product spaces.
Now it remains to prove the estimate of $\|\psi\|_{C^{1,\al}(S_{0}\ba B_{3})}$.
Multiplying (\ref{psi}) by $\rho_{4}(x)$ we obtain, on using (\ref{rhoU}) and noting
the fact that the integral operators mapping $C^{0,\al}$ functions
into $C^{1,\al}$ functions are bounded, that
\be\no
\|\psi\|_{1,\al, S_{0}\ba B_{3}}&\leq&\|\rho_{4}\psi\|_{1,\al}\\ \no
&\leq& C[\|\rho_{4}(\la_0K_{0,0}-K_{0,1})\psi\|_{1,\al}
  +\|\rho_{4}(S_{0,0}-S_{0,1})\phi\|_{1,\al}\\ \no
&&+\|\rho_{4}(\wi{K}_{0,0}-i\eta \wi{S}_{0,0})\chi\|_{1,\al}
  +\|\rho_{4}(\wi{S}_{0,1}+i\eta\wi{K}_{0,1}\wi{S}^{2})\varphi\|_{1,\al}
  +\|\rho_{4}f\|_{1,\al}]\\ \no
&\leq& C[\|U\|_{C_{0}}+\|(1-\rho_{3})U\|_{0,\al}+\|\rho_{4}f\|_{1,\al}]\\ \no
&\leq& C[\|U\|_{C_{0}}+\|\psi\|_{C^{0,\al}(S_{0}\ba B_{3})}+\|\phi\|_{C^{0,\al}(S_{0}\ba B_{3})}
 +\|\chi\|_{C^{0,\al}(\G_{0})}\\ \label{psi1al}
&&+\|\varphi\|_{C^{0,\al}(\G_{1})}+\|f\|_{1,\al,S_{0}\ba B_{3}}].
\en
Combining (\ref{C0es}) and (\ref{U0al})-(\ref{psi1al}) yields the desired estimate (\ref{4-tuple}).

We now estimate $\ds\left\|\frac{\pa v}{\pa\nu}\right\|_{0,\al,S_{0}\ba B_{2}}$.
From (\ref{v}) and the jump relation it is seen that on $S_0,$
\be\label{pavS0}
\frac{\pa v}{\pa\nu}=\frac{1}{2}\phi+T_{0,1}\psi+K^{'}_{0,1}\phi
   +(\wi{T}_{0,0}-i\eta \wi{K}^{'}_{0,0})\chi
   +(\wi{K}^{'}_{0,1}+i\eta\wi{T}_{0,1}\wi{S}^{2})\varphi.
\en
Writing $\psi=\rho_{3}\psi+(1-\rho_{3})\psi$ and$\phi=\rho_{3}\phi+(1-\rho_{3})\phi$,
we obtain from (\ref{pavS0}) that
\ben
\left\|\frac{\pa v}{\pa\nu}\right\|_{0,\al,S_{0}\ba B_{2}}
&\leq& \left\|\rho_{4}\frac{\pa v}{\pa\nu}\right\|_{0,\al,S_{0}}\\
&\leq& C[\|\psi\|_{C_{0}(S_{0})}+\|\phi\|_{C_{0}(S_{0})}+\|(1-\rho_{3})\psi\|_{1,\al,S_{0}}\\
&&+\|(1-\rho_{3})\phi\|_{0,\al,S_{0}}
  +\|\chi\|_{C^{0,\al}(\G_{0})}+\|\varphi\|_{C^{0,\al}(\G_{1})}]\\
&\leq& C[\|\psi\|_{C_{0}(S_{0})}+\|\phi\|_{C_{0}(S_{0})}+\|\psi\|_{1,\al,S_{0}\ba B_{3}}
 +\|\phi\|_{0,\al,S_{0}\ba B_{3}}\\
&&+\|\chi\|_{C^{0,\al}(\G_{0})}+\|\varphi\|_{C^{0,\al}(\G_{1})}].
\enn
Combining this with (\ref{C0es}) and (\ref{4-tuple}) yields
\be\label{paves}
\left\|\frac{\pa v}{\pa\nu}\right\|_{0,\al,S_{0}\ba B_{2}}
\leq C(\|f\|_{C_{0}(S_{0})}+\|g\|_{C_{0}(S_{0})}
 +\|f\|_{1,\al,S_{0}\ba B_{1}}+\|g\|_{0,\al,S_{0}\ba B_{1}}).
\en
This completes the proof.
\end{proof}

%%%%%%%%%%%%%%%%%%%%%%%%%%%%%%%%%%%%%%%%%%%%%%%%%%%%%%%%%%%%%%%%%%%%%%%%%%%%%%%%%%%%%%%
%%%%%%%%%%%%%%%%%%%%%%%%%%%%%%%%%%%%%%%%%%%%%%%%%%%%%%%%%%%%%%%%%%%%%%%%%%%%%%%%%%%%%%%
\section{The inverse scattering problem}\label{sec3}
\setcounter{equation}{0}

Following the ideas of \cite{KK93} for transmission problems in a homogeneous medium and
of \cite{KP} for transmission problems in an inhomogeneous medium, we prove in this section
that the interface $S_{0}$ can be uniquely determined by the far field pattern.
Combining this with the earlier result in \cite{LZH}, we have in fact proved
that both the penetrable interface $S_{0}$ and the impenetrable obstacle $\Om_{2}$
with its physical property $\mathscr{B}$ can be uniquely determined from a knowledge of
far field pattern. To establish the uniqueness result for the inverse problem,
we need the following two lemmas, in which $\Om=\R^3\ba\Om_0$ so that $\Om_{2}\subset\Om$
and $\wi{\Om}=\R^3\ba\ov{\wi{\Om}}_0$ for some domain $\wi{\Om}_0$ with the interface
$\wi{S}_0=\pa\wi{\Om}_0\cap\pa\wi{\Om}$ and with the domain $\wi{\Om}_2\subset\wi{\Om}$.

\begin{lemma}\label{lem3.1}
Suppose the positive numbers $k_{0}$, $k_{1}$ and $\la_{0}$ are given.
For $\Om_2\subset\Om,\;\wi{\Om}_2\subset\wi{\Om}$ let $G$ be the unbounded component
of $\R^3\setminus(\ov{\Om\cup\wi{\Om}})$ and let
$u^{\infty}(\hat{x},d)=\wi{u}^{\infty}(\hat{x},d)$ for all $\hat{x},\;d\in S$
with $\wi{u}^{\infty}(\hat{x},d)$ being the far field pattern of the scattered field
$\wi{u}^s(x,d)$ corresponding to the obstacle $\wi{\Om}_2$, the interface $\wi{S}_0$
and the same incident plane wave $u^i(x,d)$.
%and let $u^{\infty}(\hat{x},d)=\wi{u}^{\infty}(\hat{x},d)$ for all $\hat{x},\;d\in S$.
For $z\in G$ let $(u^{s}, v)$ be the unique solution of the problem
\be
\label{lOm0}\Delta u^{s}+k_0^2u^s=0&&\;{\rm in}\;\Om_0\setminus\{z\},\\
\label{lOm1}\Delta v+k_1^2v=0&&\;{\rm in}\;\Om_1,\\
\label{ltbc}u^s_{+}-v=-\Phi_{0}(x,z),\;\;\frac{\pa u^s_+}{\pa\nu}-\la_0\frac{\pa v}{\pa\nu}
            =-\frac{\pa\Phi_{0}(x,z)}{\pa\nu}&&\;{\rm on}\;S_0,\\
\label{lbc}\mathscr{B}(v)=0&&\;{\rm on}\;S_{1},\\
\label{lrc}\lim_{r\rightarrow\infty}r(\frac{\pa u^{s}}{\pa r}-ik_{0}u^{s})=0.
\en
Assume that $(\wi{u}^{s},\wi{v})$ is the unique solution of the problem $(\ref{lOm0})-(\ref{lrc})$
with $\Om_0,\Om_1,S_0,S_1,\mathscr{B}$ replaced by
$\wi{\Om_0},\wi{\Om_1},\wi{S_0},\wi{S_1},\wi{\mathscr{B}}$, respectively.
Then we have
\be\label{3.22}
u^s(x;z)=\wi{u}^s(x;z),\qquad x\in\ov{G}.
\en
\end{lemma}

\begin{remark}\label{r3.2} {\rm
By Theorem \ref{well-posedness}, the problem (\ref{lOm0})-(\ref{lrc}) has
a unique solution.
}
\end{remark}

\begin{proof}
By Rellich's lemma \cite{CK}, the assumption $u^{\infty}(\hat{x},d)
=\wi{u}^{\infty}(\hat{x},d)$ for all $\hat{x},\;d\in S$ implies that
 \ben
u^s(x,d)=\wi{u}^{s}(x,d),\qquad x\in G,\;d\in S.
 \enn
For the far field pattern corresponding to incident point-sources
we have by Lemma 3.3 in \cite{LZH} that
 \ben
\Phi^{\infty}(d,z)=\wi{\Phi}^{\infty}(d,z),\qquad z\in G,\;d\in S.
 \enn
Thus, Rellich's lemma \cite{CK} implies that
 \ben
u^{s}(x;z)=\wi{u}^{s}(x;z),\qquad x\in \ov{G}.
 \enn
\end{proof}

\begin{lemma}\label{lem3.3}
Assume that $f\in L^{2}(\Om_{1}), h\in C(\G_{0}),p\in C(\G_{1})$ and $g,\eta\in C(S_{0})$ with
$\eta\neq0$ and $\eta\leq0$ on $S_{0}$. Then the following problem has a unique solution
$u\in C^{2}(\Om_{1})\cap C(\pa\Om_{1})$:
\be\label{uOm1}
\Delta u+k_{1}^{2}u=f&&\qquad {\rm in}\;\;\Om_{1},\\ \label{ubg}
\frac{\pa u}{\pa\nu}+i\eta u=g&&\qquad{\rm on}\;\;S_0,\\ \label{ubh}
u=h&&\qquad{\rm on}\;\; \G_{0},\\ \label{ubp}
\frac{\pa u}{\pa\nu}+i\la u=p&&\qquad {\rm on}\;\;\G_{1}.
\en
Furthermore, there exists a constant $C>0$ such that
\ben
\|u\|_{\infty,\ov{\Om}_1}
\leq C(\|f\|_{L^2(\Om_1)}+\|g\|_{\infty,S_0}+\|h\|_{\infty,\G_0}+\|p\|_{\infty,\G_1}).
\enn
\end{lemma}

\begin{proof}
We first prove the uniqueness result, that is, $u=0$ if $f=g=h=p=0$. With the help of the
equation (\ref{uOm1}) and the boundary conditions (\ref{ubg})-(\ref{ubp}), we have
\ben
0&=&\int_{\Om_{1}}\left\{(\Delta u+k_{1}^{2}u)\ov{u}\right\}dx\\
 &=&\int_{\Om_{1}}\left\{-|\nabla u|^{2}+k_{1}^{2}|u|^{2}\right\}dx
     +\int_{S_{0}}\ov{u}\frac{\pa u}{\pa\nu}ds
     -\int_{\G_{0}}\ov{u}\frac{\pa u}{\pa\nu}ds-\int_{\G_{1}}\ov{u}\frac{\pa u}{\pa\nu}ds\\
 &=&\int_{\Om_{1}}\left\{-|\nabla u|^{2}+k_{1}^{2}|u|^{2}\right\}dx
     -i\int_{S_{0}}\eta|u|^{2}{\pa\nu}ds+i\int_{\G_{1}}\la|u|^{2}ds.
\enn
Taking the imaginary part of the above equation, we get $u=0$ on some part $\G$
of $S_{0}$ since both $\eta\neq0$ and $\eta\leq0$ on $S_{0}$ and $\la$ is a
nonnegative continuous function. By the boundary condition (\ref{ubg}) it follows that
$u={\pa u}/{\pa\nu}=0$ on $\G$. Thus, $u=0$ in $\Om_1$ by Holmgren's uniqueness theorem \cite{Kr}.

To solve the problem by means of the integral equation method we introduce the volume potential
\ben
 (Vf)(x):=\int_{\Om_{1}}\Phi_{1}(x,y)f(y)dy,\qquad x\in \Om_{1},
\enn
which defines a bounded operator $V:L^{2}(\Om_{1})\rightarrow H^{2}(\Om_{1})$
(see Theorem 8.2 in \cite{CK}). Now look for a solution in the form
\be\no
u(x)&=&-(Vf)(x)+\int_{S_{0}}\Phi_{1}(x,y)[\phi_{1}(y)+\phi_{2}(y)]ds(y)\\ \no
&&+\int_{\G_{0}}\left\{\frac{\pa\Phi_{1}(x,y)}{\pa\nu(y)}
  -i\g\Phi_{1}(x,y)\right\}(\chi_{1}(y)+\chi_{2}(y))ds(y)\\ \label{u1}
 &&+\int_{\G_{1}}\left\{\Phi_{1}(x,y)+i\g\frac{\pa\Phi_{1}(x,y)}{\pa\nu(y)}\wi{S}^{2}\right\}
   (\varphi_{1}(y)+\varphi_{2}(y))ds(y),\;\; x\in\Om_1
\en
with six densities $\phi_1\in C(S_0)$, $\phi_2\in H^{\frac{1}{2}}(S_0)$, $\chi_1\in C(\G_0)$,
$\chi_2\in H^{\frac{1}{2}}(\G_0)$, $\varphi_1\in C(\G_1)$, $\varphi_2\in H^{\frac{1}{2}}(\G_1)$
and a real coupling parameter $\g\neq0$. By $\wi{S}$ we denote the single-layer
operator (\ref{wiSij}) in the potential theoretic limit case $k_{1}=0$.

Then from the jump relations we see that the potential $u$ given by (\ref{u1})
solve the boundary value problem $(\ref{uOm1})-(\ref{ubp})$
provided the six densities satisfy the following system of integral equations:
\be\label{phi1}
\phi_{1}-2(K^{'}_{0,1}-i\eta S_{0,1})\phi_{1}
   -M_{0}(\chi_{1}+\chi_{2})-N_{0}(\varphi_{1}+\varphi_{2})
   =-2g\;\;&&\mbox{on}\;\;S_{0},\\ \label{phi2}
\phi_{2}-2(K^{'}_{0,1}-i\eta S_{0,1})\phi_{2}
   =-2\left(\frac{\pa }{\pa\nu}+i\eta\right)(Vf)\;\;&&\mbox{on}\;\;S_{0},\\ \label{chi1}
\chi_{1}+2S_{1,1}(\phi_{1}+\phi_{2})+M_{1}\chi_{1}
   +N_{1}(\varphi_{1}+\varphi_{2})=2h\;\;&&\mbox{on}\;\;\G_{0},\\ \label{chi2}
\chi_{2}+M_{1}\chi_{2}=2Vf\;\;&&\mbox{on}\;\;\G_{0},\\ \label{varphi1}
\varphi_{1}-2(K^{'}_{1,1}+i\la S_{1,1})(\phi_{1}+\phi_{2})-M_{2}(\chi_{1}+\chi_{2})
   -N_{2}\varphi_{1}=-2p\;\;&&\mbox{on}\;\;\G_{1},\\ \label{varphi2}
\varphi_2-N_1\varphi_2=-2\left(\frac{\pa}{\pa\nu}+i\la\right)(Vf)\;\;&&\mbox{on}\;\;\G_{1},
\en
where
\ben
M_0&=&2(\wi{T}_{0,0}-i\g\eta \wi{K}^{'}_{0,0}+i\eta\wi{K}_{0,0}+\g\eta\wi{S}_{0,0}),\\
N_0&=&2(\wi{K}^{'}_{0,1}+i\g\wi{T}_{0,1}\wi{S}^2+i\eta\wi{S}_{0,1}-\g\eta\wi{K}_{0,1}\wi{S}^2),\\
M_1&=&2(\wi{K}_{1,0}-i\g \wi{S}_{1,0}),\\
N_1&=&2(\wi{S}_{1,1}+i\g\wi{K}_{1,1}\wi{S}^{2}),\\
M_2&=&2(\wi{T}_{1,0}-i\g\wi{K}^{'}_{1,0}+i\la\wi{K}_{1,0}+\la\g\wi{S}_{1,0}),\\
N_2&=&2(\wi{K}_{1,1}+i\g\wi{T}_{1,1}\wi{S}^{2}-\frac{1}{2}\la\g\wi{S}^{2}
      +i\la\wi{S}_{1,1}-\la\g\wi{K}_{1,1}\wi{S}^{2}).
\enn

Precisely, we seek a solution $(u,\phi_1,\phi_2,\chi_1,\chi_2,\varphi_1,\varphi_2)\in
C(\ov{\Om_1})\times C(S_0)\times H^{\frac12}(S_0)\times C(\G_0)\times H^{\frac12}(\G_0)
\times C(\G_1)\times H^{\frac12}(\G_1):=Y$ to the system of integral equations
(\ref{u1})-(\ref{varphi2}). Similarly as in the proof of Theorem \ref{well-posedness},
it can be shown by using the Riesz-Fredholm theorem that this system has unique solution
in the space $Y$ and there exists a positive constant $C>0$ such that
%$C(\ov{\Om_{1}})\times C(S_{0})\times H^{\frac{1}{2}}(S_{0})
%\times C(\G_{0})\times H^{\frac{1}{2}}(\G_{0})\times C(\G_{1})\times H^{\frac{1}{2}}(\G_{1})$.
%Therefore, the Riesz-Fredholm theorem is applicable again and consequently
%there exists a positive constant $C>0$ such that
\ben
\|u\|_{\infty,\ov{\Om_{1}}}&\leq& C(\|Vf\|_{\infty,\ov{\Om_{1}}}+\|g\|_{\infty, S_{0}}
 +\left\|\left(\frac{\pa }{\pa\nu}+i\eta\right)(Vf)\right\|_{H^{\frac12}(S_0\cup S_1)}\\
&&+\|h\|_{\infty, \G_{0}}+\|p\|_{\infty,\G_1})\\
&\leq& C(\|f\|_{L^{2}(\Om_{1})}+\|g\|_{\infty, S_{0}}+\|h\|_{\infty, \G_{0}}+\|p\|_{\infty, \G_{1}}).
\enn
The lemma is thus proved.
\end{proof}

We are now in a position to state and prove the main result of this section.

\begin{theorem}\label{thm3}
Suppose the positive numbers $k_{0}$, $k_{1}$ and $\la_{0}$ ($\la_{0}\neq1$) are given.
Assume that $S_{0}$ and $\wi{S}_{0}$ are two penetrable interfaces and
$\Om_{2}$ and $\wi{\Om}_{2}$ are two impenetrable obstacles with boundary conditions
$\mathscr{B}$ and $\wi{\mathscr{B}}$, respectively, for
the corresponding scattering problem. If the far field patterns of the scattered fields
for the same incident plane wave $u^{i}(x)=e^{ik_{0}x\cdot d}$ coincide at a fixed frequency
for all incident direction $d\in S$ and observation direction $\wi{x}\in S$,
then $S_0=\wi{S}_0$, $\Om_2=\wi{\Om}_2$ and $\mathscr{B}=\wi{\mathscr{B}}$.
\end{theorem}

\begin{proof}
We just need to prove that $S_{0}=\wi{S}_{0}$ since the remanning part of the theorem
then follows from this and Theorem 3.7 in \cite{LZH}. Let $G$ be defined as in Lemma \ref{lem3.1}.
Assume that $S_{0}\neq\wi{S}_{0}$. Then, without loss of generality,
we may assume that there exists $z_{0}\in S_{0}\ba\ov{\wi{\Om}}$.
Let $B_{2}$ be a small ball centered at $z_{0}$ such that $B_{2}\cap\ov{\wi{\Om}}=\emptyset$.
Choose $h>0$ such that the sequence
\ben
 z_j:=z_0+\frac{h}{j}\nu(z_{0}), \qquad j=1,2,\ldots,
\enn
is contained in $G\cap B_{2}$, where $\nu(z_{0})$ is the outward normal to
$S_{0}$ at $z_{0}$. Using the notations in Lemma \ref{lem3.1} and letting
$(u^{s}_{j},v_{j})$ and $(\wi{u}^{s}_{j},\wi{v}_{j})$ be the solutions of (\ref{lOm0})-(\ref{lrc})
with $z=z_j$. Then, by Lemma \ref{lem3.1}, $u^{s}_{j}=\wi{u}^{s}_{j}:=u_{j}$ in $\ov{G}$.
Since $z_{0}$ has a positive distance from $\ov{\wi{\Om}}$,
we conclude from the well-posedness of the direct scattering problem that
there exists $C>0$ such that
\be\label{ujes}
\|u_j\|_{\infty,S_0\cap B_2}+\left\|\frac{\pa u_j}{\pa\nu}\right\|_{\infty,S_0\cap B_2}
\leq C\qquad\mbox{for all}\;\; j\geq1.
\en
Choose a small ball $B_{1}$ with center $z_{0}$ which is strictly contained in $B_{2}$. Since
\ben
\|\Phi_{0}(\cdot,z_j)\|_{C_{0}(S_{0})}+\|\Phi_{0}(\cdot,z_j)\|_{1,\al,S_{0}\ba B_{1}}&\leq& C,\\
\|\frac{\pa\Phi_{0}}{\pa\nu}(\cdot,z_j)\|_{C_{0}(S_{0})}
   +\|\frac{\pa\Phi_{0}}{\pa\nu}(\cdot,z_j)\|_{0,\al,S_{0}\ba B_{1}}&\leq& C
\enn
for some positive constant independent of $j$, we conclude from Lemma \ref{lemapriori} that
\ben
\|v_j\|_{\infty,S_0\ba B_2}+\left\|\frac{\pa v_j}{\pa\nu}\right\|_{\infty,S_0\ba B_2}
\leq C\qquad\mbox{for all}\;\; j\geq1.
\enn
From this it follows that
\be\label{vj-Phi}
\|\la_{0}v_j-\Phi_0(\cdot,z_j)\|_{\infty,S_0\ba B_2}
    &\leq& C\qquad\mbox{for all}\;\;j\geq1,\\ \label{pavj-Phi}
\left\|\la_0\frac{\pa v_j}{\pa\nu}-\frac{\pa\Phi_0(\cdot,z_j)}{\pa\nu}
    \right\|_{\infty,S_0\ba B_2}&\leq& C\qquad\mbox{for all}\;\; j\geq1.
\en
The transmission boundary conditions yield
\be\label{vj-Phi2}
\|v_j-\Phi_{0}(\cdot,z_j)\|_{\infty, S_{0}\cap B_{2}}=\|u_j\|_{\infty, S_{0}\cap B_{2}}
   &\leq& C\qquad \mbox{for all}\;\; j\geq1,\\ \label{pavj-Phi2}
\left\|\la_{0}\frac{\pa v_{j}}{\pa\nu}-\frac{\pa \Phi_{0}(\cdot,z_j)}{\pa\nu}
   \right\|_{\infty, S_0\cap B_2}=\left\|\frac{\pa u_j}{\pa\nu}\right\|_{\infty,S_0\cap B_2}
   &\leq& C\qquad \mbox{for all}\;\; j\geq1.
\en
Combining (\ref{pavj-Phi}) and (\ref{pavj-Phi2}) yields
\be\label{pavj-Phi3}
\left\|\la_{0}\frac{\pa v_{j}}{\pa\nu}-\frac{\pa\Phi_0(\cdot,z_j)}{\pa\nu}
\right\|_{\infty, S_{0}}\leq C\qquad\mbox{for all}\;\; j\geq1.
\en
This can be used together with (\ref{vj-Phi}) to prove the estimate
\be\label{vj-Phi3}
\|\la_{0}v_j-\Phi_{0}(\cdot,z_j)\|_{\infty,S_{0}}\leq C \qquad \mbox{for all}\;\; j\geq1.
\en
In fact, choose a non-positive function $\eta\in C^{2}(S_{0}),$ $\eta\not\equiv0$ and
supported in $S_{0}\ba B_{2}$. Then $w_{j}:=\la_{0}v_j-\Phi_{0}(\cdot,z_j)$ solves
the following boundary value problem:
\ben
\Delta w_j+k_1^2w_j=(k^2_0-k^2_1)\Phi_0(\cdot,z_j)&&\qquad\mbox{in}\;\;\Om_1,\\
\frac{\pa w_j}{\pa\nu}+i\eta w_j=(\frac{\pa}{\pa\nu}+i\eta)[\la_0v_j-\Phi_0(\cdot,z_j)]
  &&\qquad\mbox{on}\;\;S_0,\\
w_j=-\Phi_0(\cdot,z_j)&&\qquad\mbox{on}\;\;\G_{0},\\
\frac{\pa w_j}{\pa\nu}+i\la w_j=-(\frac{\pa}{\pa\nu}+i\la)\Phi_{0}(\cdot,z_j)
  &&\qquad \mbox{on}\;\; \G_{1}.
\enn
Since, by (\ref{vj-Phi}) and (\ref{pavj-Phi3}),
$f:=(k^{2}_{0}-k^{2}_{1})\Phi_{0}(\cdot,z_j)\in L^{2}(\Om_{1}),$
$h:=-\Phi_{0}(\cdot,z_j)\in C(\G_{0}),$
$p:=-(\frac{\pa }{\pa\nu}+i\la)\Phi_{0}(\cdot,z_j)\in C(\G_{1})$ and
$g:=(\frac{\pa }{\pa\nu}+i\eta) [\la_{0}v_j-\Phi_{0}(\cdot,z_j)]\in C(S_{0})$,
then the desired result (\ref{vj-Phi3}) follows from Lemma \ref{lem3.3}.

Now the triangle inequality together with (\ref{vj-Phi2}) and (\ref{vj-Phi3}) implies that
\ben
\|(\la_{0}-1)\Phi_{0}(\cdot,z_j)\|_{\infty, S_{0}\cap B_{2}}
 &\leq& \|\la_{0}\Phi_{0}(\cdot,z_j)-\la_{0}v_j\|_{\infty, S_{0}\cap B_{2}}
   + \|\la_{0}v_j-\Phi_{0}(\cdot,z_j)\|_{\infty, S_{0}\cap B_{2}}\\
 &\leq& C.
\enn
This is a contradiction since $\la_0\neq1$ and
$|\Phi_0(z_0,z_j)|_{\infty,S_0\cap B_2}\rightarrow\infty$ as $j\rightarrow\infty.$
The proof is thus complete.
\end{proof}

\begin{remark}{\rm
Our method can be extended straightforwardly to both the 2D case and the case of a
multilayered medium, and a similar result can be obtained (that is, all the interfaces
between the layered media as well as the embedded obstacle can be uniquely determined).
}
\end{remark}

\section*{Acknowledgements}
This work was supported by the NNSF of China under grant No. 10671201.

%\newpage

\end{document}